\documentclass[12pt,xcolor=dvipsnames]{amsart}
\usepackage{url}
\usepackage{pdfsync}
\usepackage{amsmath,amsfonts,amssymb}
\usepackage{amscd}
\usepackage{graphicx}
\usepackage[all]{xy}
\usepackage{mathrsfs}
\usepackage{colordvi,color}
\newcommand{\N}{\mathbb{N}}
\newcommand{\Z}{\mathbb{Z}}

\usepackage[all]{xy}
\usepackage{tikz}
\usetikzlibrary{arrows,decorations.pathmorphing,decorations.pathreplacing,positioning,shapes.geometric,shapes.misc,decorations.markings,decorations.fractals,calc,patterns}

\usepackage{graphicx}

\numberwithin{equation}{section}

\newcommand{\cD}{\mathcal{D}}

\newcommand{\cP}{\mathcal{P}}
\newcommand{\al}{\alpha}
\newcommand{\be}{\beta}
\newcommand{\ga}{\gamma}
\newcommand{\de}{\delta}
\newcommand{\eps}{\varepsilon}
\DeclareMathOperator{\SL}{SL}
\DeclareMathOperator{\GL}{GL}
\DeclareMathOperator{\Pol}{Pol}

\newtheorem{Lemma}{Lemma}[section]
\newtheorem{Theorem}[Lemma]{Theorem}
\newtheorem{Proposition}[Lemma]{Proposition}
\newtheorem{Corollary}[Lemma]{Corollary}
\newtheorem{cor}[Lemma]{Corollary}

\theoremstyle{definition}
\newtheorem{Definition}[Lemma]{Definition}
\newtheorem{Remark}[Lemma]{Remark}
\newtheorem{Remarks}[Lemma]{Remarks}
\newtheorem{Example}[Lemma]{Example}

\begin{document}

\setlength{\parindent}{0pt}
\setlength{\parskip}{2pt}

\title[Dissected polygons and generalized frieze patterns]{
Conway-Coxeter friezes and beyond: \\
Polynomially weighted walks around dissected polygons
and generalized frieze patterns}

\author{Christine Bessenrodt}
\address{Institut f\"{u}r Algebra, Zahlentheorie und Diskrete
  Mathematik, Fa\-kul\-t\"{a}t f\"{u}r Ma\-the\-ma\-tik und Physik, Leibniz
  Universit\"{a}t Hannover, Welfengarten 1, 30167 Hannover, Germany}
%\email{bessen@math.uni-hannover.de}
%\urladdr{http://www.iazd.uni-hannover.de/\~{ }bessen}

\thanks{November 23, 2014}



\begin{abstract}
Conway and Coxeter introduced frieze patterns
in 1973 and classified them via triangulated polygons.
The determinant of the matrix
associated to a frieze table was
computed explicitly by Broline, Crowe
and Isaacs in 1974, a result generalized
2012 by Baur and Marsh in the context of cluster algebras
of type~A.
Higher angulations of polygons and associated generalized
frieze patterns were studied in a joint paper with Holm and
J\o rgensen.
Here we take these results further; we allow arbitrary dissections and
introduce polynomially weighted walks around such dissected polygons.
The corresponding generalized frieze table satisfies a
complementary symmetry condition;
its  determinant is a multisymmetric multivariate polynomial
that is given explicitly.
But even more, the frieze matrix may be transformed over a
ring of Laurent polynomials to a nice diagonal form
generalizing the Smith normal form result given in \cite{BHJ}.
Considering the generalized polynomial frieze in this context
it is also shown that the non-zero local determinants
are monomials that are given explicitly, depending on the geometry of
the dissected polygon.
\end{abstract}

\maketitle

%%%%%%%%%%%%%%%%%%%%%%%%%%%%%%%%%%%%%%%%%%%%%%%%%%%%%%%%%%%%
\section{Introduction}
\label{sec:introduction}

Conway and Coxeter \cite{CC1,CC2} defined arithmetical friezes
by a local determinant condition and classified them via
triangulated polygons; the friezes have a glide reflection symmetry.
The geometry of these triangulations was then studied by
Broline, Crowe and Isaacs \cite{BCI};
they associated a matrix to these which was shown to be symmetric,
and they computed its determinant to be $-(-2)^{n-2}$, for
any triangulated $n$-gon.

These notions have found renewed interest in recent years in the
context of cluster algebras, see for example \cite{ARS, BM, CaCh}.
Indeed, Baur and Marsh \cite{BM} generalized the determinant
formula in the context of cluster algebras of type~$A$,
to a version involving cluster variables.

Another generalization and refinement of the results on
Conway-Coxeter friezes
was investigated recently in joint work with Holm and J\o rgensen
on $d$-angulations of polygons \cite{BHJ}.
Again, a matrix associated to such $d$-angulations with entries
counting suitable paths turned out to be symmetric,
and its determinant and Smith normal forms were computed;
these are independent of the
particular $d$-angulation.
We have also generalized the notion of friezes in this context
by weakening the local determinant condition.
It should also be pointed out that
in recent work by Holm and J\o rgensen,
generalized friezes were categorified via a modified
Caldero-Chapoton map \cite{HJ1, HJ2}.

Here the investigations of \cite{BHJ} are taken
further in two combinatorial directions.
On the one hand, instead of $d$-angulations we study arbitrary
polygon dissections; the results from \cite{BHJ}
have natural generalizations in this wider context,
Furthermore, going far beyond the arithmetical results,
we do not only count suitable paths along the polygons,
but we consider weighted walks around the dissected polygons.

The corresponding associated weight matrices then have
(multivariate) polynomial entries;
they are not quite symmetric, but
satisfy a natural complementarity condition
(see Theorem~\ref{thm:complementarysymmetry} and Corollary~\ref{cor:complementarysymmetry}).
Surprisingly, as seen in the arithmetical situation,
the determinant only depends on the pieces of the dissection but
not on the specific way the pieces are glued together in the dissection;
an explicit formula for the determinant of a weight matrix
is given in Theorem~\ref{thm:detandsmith}.
In fact, we are even able to provide an equivalent diagonal form
of the weight matrix over a suitable ring of Laurent polynomials.

As in the arithmetical situation, we associate a
generalized frieze pattern to a dissected polygon;
this is the frieze of weight polynomials, which only
has a translation symmetry due to the complementarity condition.
Furthermore, the local $2\times 2$ determinants 0 and 1
in the generalized arithmetical friezes
are seen to be specializations of 0's and explicit
monomials in the generalized polynomial friezes.
The condition when we have a nonzero determinant and
the explicit description of the monomials
are given in Theorem~\ref{thm:zigzag}.

\section{From arcs around triangulated polygons to walks around dissected polygons}

A frieze in the sense of Conway and Coxeter \cite{CC1,CC2} is
a pattern of natural numbers arranged in bi-infinite interlaced
rows as in the example further  below where the top and bottom rows consist
only of 1's and every set of four adjacent numbers located in a diamond
$$\begin{array}{ccc}
&b& \\ a&&d\\&c&
\end{array}$$
satisfies the determinant condition $ad-bc=1$.
We say that the frieze pattern is of
{\em height}~$n$ if it has $n$ bi-infinite rows.
For example, here is a frieze pattern of height~4:
$$
\begin{array}{cccccccccccccccccccccccc}
\ldots & 1 & & 1 & & 1 & & 1 & & {\bf 1} & & 1 & & 1 & & 1 & & 1 & & 1 & & 1 & \ldots \\
& \ldots & 1 & & 3 & & 1 & & {\bf 2} & & {\bf 2} & & 1 & & 3 & & 1 & & 2 & & 2 &
\ldots & & \\
\ldots & 1 & & 2 & & 2 & & {\bf 1} & & {\bf 3} & & {\bf 1} & & 2 & & 2 & & 1 & & 3 & & 1 & \ldots \\
& \ldots & 1 & & 1 & & {\bf 1} & & {\bf 1} & & {\bf 1} & & {\bf 1} & & 1 & & 1 & & 1 & & 1 &
\ldots & & \\
\end{array}
$$
In bold, a {\em fundamental region} of the frieze pattern
is marked; as was noticed by Conway and Coxeter,
frieze patterns are periodic and they arise
from a fundamental region by  applying glide reflections.

As their main result, they showed that frieze patterns
can be classified via triangulations of (convex) polygons.
More precisely, a triangulated polygon with $n+1$ vertices
corresponds to a frieze pattern of height $n$ whose second row
is obtained by walking around the polygon and counting at  each
vertex the number of incident triangles. The other rows are then
easily computed using the determinant condition.

For example, the triangulated pentagon corresponding
to the frieze pattern above is pictured here
(with the vertices numbered counterclockwise):
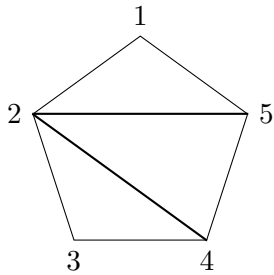
\[
  \begin{tikzpicture}[auto]
    \node[name=s, shape=regular polygon, regular polygon sides=5, minimum size=3cm, draw] {};
    \draw[thick] (s.corner 2) to (s.corner 4);
    \draw[thick] (s.corner 2) to (s.corner 5);
    \draw[shift=(s.corner 1)]  node[above]  {{\small 1}};
  \draw[shift=(s.corner 2)]  node[left]  {{\small 2}};
  \draw[shift=(s.corner 3)]  node[below]  {{\small 3}};
  \draw[shift=(s.corner 4)]  node[below]  {{\small 4}};
  \draw[shift=(s.corner 5)]  node[right]  {{\small 5}};
  \end{tikzpicture}
\]
Counting the incident triangles at the vertices (starting at vertex~1
and going counterclockwise)
gives the sequence
$1$, $3$, $1$, $2$, $2$, whose repetition leads
to the second row in the frieze pattern above.

\smallskip

Broline, Crowe and Isaacs \cite{BCI} gave a geometrical
interpretation also for the other rows of the frieze pattern
by counting suitable sequences of distinct triangles
for any two vertices, called {\em arcs}.
This was generalized in the context of $d$-angulations
to the notion of {\em $d$-paths}~\cite{BHJ}.

It turned out that this notion was also the right concept to generalize
a determinant formula for the matrix of arc numbers of a
triangulated polygon in \cite{BCI} to one for the matrix of $d$-path
numbers of a $d$-angulated polygon in \cite{BHJ}; moreover, the result
was refined in \cite{BHJ} so as to give also the Smith normal form
of the matrix.
Furthermore, putting the $d$-path numbers into a frieze produced
a generalized frieze pattern where all
the local $2\times 2$-determinants are 0 or~1, and
the position of the 1's is explained geometrically
from the given $d$-angulation.

\medskip

In the following, arbitrary dissected convex polygons are considered
and in this context we generalize the notion of $d$-paths
in $d$-angulated polygons (and thus of arcs of triangulated polygons).
We first introduce some notation and we
take a brief look at the arithmetical results,
and in the next section we take the main step and refine the
arithmetics of counting walks to a polynomially weighted version.
\medskip

We consider a convex $n$-gon $P_n$,
dissected by $m-1$
(pairwise noncrossing) diagonals into $m$ polygons
$\al_1, \ldots, \al_m$ with $d_1, \ldots, d_m$ vertices,
respectively; these polygons are called the {\em pieces}
of the dissection~$\cD=\{\al_1,\ldots,\al_m\}$.
A $d$-gon will also be called a  {\em polygon of degree~$d$};
for a piece $\al$ we also write $d_\al$ for its degree.
The multiset of degrees of the pieces of a dissection $\cD$ is
called its {\em type}.

We label the vertices of $P_n$
by $1,2,\ldots,n$ in counterclockwise order;
this numbering is occasionally extended and then
taken modulo~$n$ below.

\begin{Definition} \label{def:walk}
Let $n\ge 3$ and $m\in \mathbb{N}$.
Let $\cD$ be a dissection of
$P_n$ into $m$ polygons as above.
Let $i$ and $j$ be vertices of
$P_n$.

A {\em (counterclockwise) walk $w$ (around $P_n$) from $i$ to $j$}
is a sequence
$$w=(p_{i+1},p_{i+2},\ldots,
p_{j-2},p_{j-1})$$
of
pieces of the dissection $\cD$
that satisfies the following properties:
\begin{itemize}
\item[{(i)}] The polygon $p_k$ is incident to vertex~$k$,
for all $k\in \{i+1, i+2,\ldots,j-1\}$.
\item[{(ii)}]  For any $d$, a polygon $\al \in \cD$ of degree~$d$
appears at most $d-2$ times among the
pieces $p_{i+1},p_{i+2},
\ldots, p_{j-2},p_{j-1}$.
\end{itemize}
Counting the walks gives a matrix $M_\cD=(m_{ij})$, where
$m_{ij}$ is the number of walks from vertex $i$ to vertex~$j$.
\end{Definition}

Thus, a walk from $i$ to $j$ around a dissected polygon $P_n$
moves counterclockwise from $i$ to $j$ around the $n$-gon
and at each intermediate vertex picks an incident polygon,
such that pieces of degree $d$ are chosen at most
$d-2$ times.
For $i=j$, there is no walk from $i$ to~$i$, while for $j=i+1$,
the empty sequence gives exactly one walk from $i$ to~$i+1$.

Note that in the case of a $d$-angulation $\cD$ of the polygon $P_n$
the walks are exactly the $d$-paths introduced in~\cite{BHJ}.
In particular, for $d=3$ this
provides a further generalization of the definition of arcs
in \cite{BCI} for triangulations.

\begin{Example}
For an illustration of Definition~\ref{def:walk},
we consider the following
dissection of a 7-gon into three triangles and a quadrangle:
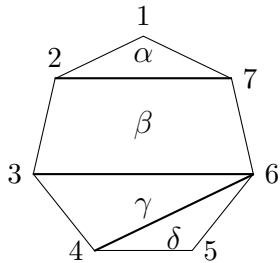
\[
  \begin{tikzpicture}[auto]
    \node[name=s, shape=regular polygon, regular polygon sides=7, minimum size=3cm, draw] {};
 \draw[thick] (s.corner 2) to node[above=2pt] {$\al$} (s.corner 7);
 \draw[thick] (s.corner 3) to node[above=9pt] {$\be$}
 node[below=5pt] {$\ga$} (s.corner 6);
 \draw[thick] (s.corner 4) to node[below=1pt] {$\de$}
 (s.corner 6);
  \draw[shift=(s.corner 1)]  node[above]  {{\small 1}};
  \draw[shift=(s.corner 2)]  node[above]  {{\small 2}};
  \draw[shift=(s.corner 3)]  node[left]  {{\small 3}};
  \draw[shift=(s.corner 4)]  node[left]  {{\small 4}};
  \draw[shift=(s.corner 5)]  node[right]  {{\small 5}};
  \draw[shift=(s.corner 6)]  node[right]  {{\small 6}};
  \draw[shift=(s.corner 7)]  node[right]  {{\small 7}};
  \end{tikzpicture}
\]

By Definition~\ref{def:walk}, no $d$-gon is allowed to appear more than $d-2$ times in a walk. We determine the walks from vertex~1 to any vertex $j>2$.

Let $j=3$; at vertex $2$ we may choose $\al$ or $\be$, hence the walks are $(\al)$ and  $(\be)\}$.
For $j=4$, we may choose $\al$ or $\be$ at
vertex~$2$, and in both cases the choices $\be$ or $\ga$ at vertex~$3$
are possible (as $\be$ is a 4-gon). Hence we have the walks $(\al,\be), (\al,\ga),(\be,\be),(\be,\ga)$.
Next, let $j=5$; the walks just listed can all be extended by $\de$,
and we also have the walks $(\al,\be,\ga),(\be,\be,\ga)$.
For $j=6$, we note that at vertex~$5$ we have to choose $\de$,
and since $\de$ may only appear once in a walk, this leaves only two walks
$(\al,\be,\ga,\de),(\be,\be,\ga,\de)$ from vertex~1 to vertex~6.
For $j=7$, the only possible walk is $(\al,\be,\ga,\de,\be)$.

We will see later that there is indeed always only one walk from vertex
$i$ to vertex $i-1$ (labels taken mod~$n$).

In our example, we have just computed the first row of the
counting matrix $M_{\cD}$, which is in full given here:
$$
M_{\cD}=
\begin{pmatrix}
0&1&2&4&6&2&1\\
1&0&1&2&3&1&1\\
2&1&0&1&2&1&1\\
4&2&1&0&1&1&2\\
6&3&2&1&0&1&3\\
2&1&1&1&1&0&1\\
1&1&1&2&3&1&0
\end{pmatrix}
$$
This is a symmetric matrix with  determinant
$$\det M_\cD = 24\:,$$
and furthermore,
there are matrices $U,V \in \SL(7,\Z)$ such that
$$UM_\cD V =  \Delta(3,2,2,2,1,1,1)\:,$$
where we denote by $\Delta(a_1,\ldots,a_n)$
a diagonal matrix with diagonal entries $a_1,\ldots,a_n$.
\end{Example}

For arbitrary dissections, we have
the following result as a consequence of later theorems;
these arithmetical properties generalize the results in~\cite{BHJ} for
$d$-angulations and thus in particular the results in~\cite{BCI} for triangulated polygons.
Observe that the determinant and even the diagonal form do not depend
on the specific dissection but only on the sizes of its pieces.

\begin{Theorem}\label{thm:arithmetic}
Let $n\in \N$, $n\geq 3$.
Let $\cD=\{\al_1,\ldots,\al_m\}$ be a dissection of~$P_n$
of type $\{d_1,\ldots,d_m\}$, and let $M_\cD$ be the corresponding
matrix.
Then the matrix $M_\cD$ is symmetric, with
determinant
$$\det M_\cD = (-1)^{n-1} \prod_{j=1}^m (d_j-1) \:.$$
Furthermore, there are matrices $U,V \in \GL(n,\Z)$ with
$$UM_\cD V = \Delta(d_1-1,\ldots,d_m-1,1,\ldots,1)\:.$$
\end{Theorem}
\medskip

These arithmetical assertions may also be obtained by analyzing
and suitably modifying some of the proofs in~\cite{BHJ}.
Note that in relation to the associated frieze pattern, the symmetry of the
matrix corresponds to the glide reflection symmetry of the frieze.

%%%%%%%%%%%%%%%%%%%%%%%%%%%%%%%%%%%%%%%%%%%%%%%%%%%%%%%%%%%
\section{Weighted walks around dissected polygons}

In this section  we refine the
arithmetics of counting suitable paths (as in \cite{BHJ})
to a weighted version for walks.

We keep the notation from the previous section, i.e.,
we consider a convex $n$-gon~$P_n$ with a
dissection~$\cD=\{\al_1,\ldots,\al_m\}$
of type $\{d_1,\ldots,d_m\}$,
and  the vertices are labelled
$1,2,\ldots,n$ in counterclockwise order.

\begin{Definition} \label{def:weights}
To each polygon $\al_l$ in $\cD$
we associate an indeterminate $x_l=x(\al_l)$; occasionally we will also write $x_\al$ for the indeterminate
to a polygon $\al$ in $\cD$.
We define the {\em weight} of a (counterclockwise)
walk $w= (p_{i+1},p_{i+2},\ldots,
p_{j-2},p_{j-1})$ from $i$ to $j$
to be the monomial
$$x^w=\prod_{k=i+1}^{j-1} x(p_k)\:.
$$
Then we define the $n\times n$ {\em weight matrix}
$W_\cD(x)=(w_{i,j})_{1\leq i,j\leq n}$ associated to
the dissection~$\cD$
by setting its $(i,j)$-entry to be the polynomial
$$
w_{i,j} =\sum_{w: \text{ walk from $i$ to $j$}} x^w \in \Z[x_1,\ldots,x_m]=\Z[x]
\:.$$
\end{Definition}

\begin{Remarks} \label{rem:specialweights}
(i) For the polynomials $w_{i,j}$ associated
to a dissection $\cD$ of $P_n$
we have the properties:
$w_{i,i}=0$ and $w_{i,i+1}=1$,
where the only walk from $i$ to $i+1$ (the empty sequence)
is weighted by~1.

(ii) If the dissection $\cD$ of $P_n$ is of type
$\{d_1, \ldots, d_m\}$, then $n=2+\sum_{k=1}^m (d_k-2)$.
As any walk from $i+1$ to $i$ (around $P_n$) is of length $n-2$,
it has to contain every $d$-gon of $\cD$ exactly $d-2$ times,
hence the weight of such a walk is $\prod_{k=1}^m x_k^{d_k-2}$.
\end{Remarks}

\begin{Example}
We illustrate Definition~\ref{def:weights}
with the same dissection of a 7-gon into three triangles and a quadrangle as before:
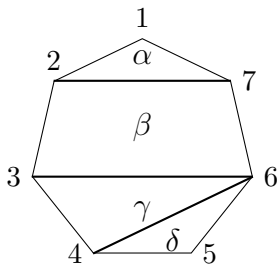
\[
  \begin{tikzpicture}[auto]
    \node[name=s, shape=regular polygon, regular polygon sides=7, minimum size=3cm, draw] {};
     \draw[thick] (s.corner 2) to node[above=2pt] {$\al$} (s.corner 7);
 \draw[thick] (s.corner 3) to node[above=9pt] {$\be$}
 node[below=5pt] {$\ga$} (s.corner 6);
 \draw[thick] (s.corner 4) to node[below=1pt] {$\de$}
 (s.corner 6);
  \draw[shift=(s.corner 1)]  node[above]  {{\small 1}};
  \draw[shift=(s.corner 2)]  node[above]  {{\small 2}};
  \draw[shift=(s.corner 3)]  node[left]  {{\small 3}};
  \draw[shift=(s.corner 4)]  node[left]  {{\small 4}};
  \draw[shift=(s.corner 5)]  node[right]  {{\small 5}};
  \draw[shift=(s.corner 6)]  node[right]  {{\small 6}};
  \draw[shift=(s.corner 7)]  node[right]  {{\small 7}};
  \end{tikzpicture}
\]
We associate to the pieces $\al$, $\be$, $\ga$, $\de$ the weight
indeterminates $a,b,c,d$, respectively.

We want to compute the weight polynomials $w_{i,j}$ for $i=1$;
we already know that $w_{1,1}=0$ and $w_{1,2}=1$, so we now take $j>2$.
In the previous section we have already  listed the
corresponding walks  explicitly, and thus
we obtain the following weights.

Let $j=3$; as the possible walks are $(\al)$ or $(\be)$, we have $w_{1,3}=a+b$.
For $j=4$,   the walks $(\al,\be), (\al,\ga),(\be,\be),(\be,\ga)$ give  $w_{1,4}=(a+b)(b+c)$.
Next, let $j=5$; the walks determined before give
$w_{1,5}=(a+b)(b+c)d+(a+b)bc$.
For $j=6$, we only had the two walks
$(\al,\be,\ga,\de),(\be,\be,\ga,\de)$, hence
$w_{1,6}=(a+b)bcd$.
For $j=7$, the only possible walk was $(\al,\be,\ga,\de,\be)$, thus
$w_{1,7}=ab^2cd$.

Thus we have computed the first row of the weight matrix $W_\cD$ given here:
$$\tiny
\begin{pmatrix}
0&1&a+b&(a+b) (b+c)&(a+b) (b+c) d+(a+b) b c&(a+b) b c d&a b^2 c d\\
a b^2 c d&0&1&b+c&b (c+d)+c d&b c d&b^2 c d\\
(a+b) b c d&a b^2 c d&0&1&c+d&c d&b c d\\
(a+b) (b+c) d&a b (b+c) d&a b^2 c d&0&1&d&(b+c) d\\
(a+b) (b+c+d)&a b (b+c+d)&a b^2 (c+d)&a b^2 c d&0&1&b+c+d\\
a+b&a b&a b^2&a b^2 c&a b^2 c d&0&1\\
1&a&a b&a b (b+c)&a b (b+c) d+a b^2 c&a b^2 c d&0
\end{pmatrix}
$$
The determinant of this matrix is computed
to be
\medskip

$1+{a}^{5}{b}^{10}{c}^{3}{d}^{3}+{a}^{2}{b}^{8}{c}^{2}{d}^{2}+{a}^{5}{b
}^{8}{c}^{3}{d}^{5}+{a}^{2}{b}^{6}{c}^{2}{d}^{4}+{a}^{6}{b}^{12}{c}^{4
}{d}^{6}+{a}^{3}{b}^{10}{c}^{3}{d}^{5}+{a}^{4}{b}^{6}{c}^{2}{d}^{2}+a{
b}^{2}{c}^{3}d+{a}^{5}{b}^{8}{c}^{5}{d}^{3}+{a}^{2}{b}^{6}{c}^{4}{d}^{
2}+{a}^{6}{b}^{12}{c}^{6}{d}^{4}+{a}^{3}{b}^{10}{c}^{5}{d}^{3}+{a}^{5}
{b}^{6}{c}^{5}{d}^{5}+{a}^{2}{b}^{4}{c}^{4}{d}^{4}+{a}^{6}{b}^{10}{c}^
{6}{d}^{6}+{a}^{3}{b}^{8}{c}^{5}{d}^{5}+{a}^{4}{b}^{12}{c}^{6}{d}^{6}+
{a}^{7}{b}^{14}{c}^{7}{d}^{7}+{a}^{4}{b}^{4}{c}^{4}{d}^{2}+a{b}^{4}cd+
{b}^{2}cd{a}^{3}+{a}^{4}{b}^{4}{c}^{2}{d}^{4}+a{b}^{2}c{d}^{3}
$.
\medskip

We observe that this determinant is symmetric in the indeterminates $a,c,d$ corresponding to the three triangles.
Indeed, we can factorize this determinant as
$$
\det W_\cD= (1+a^3b^2cd)\cdot (1+ab^2c^3d)\cdot (1+ab^2cd^3)\cdot(1+ab^4cd+(ab^4cd)^2)\:.
$$
This serves as an illustration of the determinantal result to be proved later
in Theorem~\ref{thm:detformula}.
\end{Example}

\medskip

In the arithmetical situation where all weights are specialized to~1,
the path enumerating matrix is symmetric.
In the generalized setting that we are considering now,
the weight matrix is obviously not symmetric in general (see the example above),
but it has a symmetry with respect to complementing walk weights.
This results in having only a translation symmetry of the corresponding generalized
polynomial frieze (see Section~\ref{sec:friezes}).

\medskip

First we need to define a linear map
which arises from complementing the weights of walks
with respect to the maximal possible weight.

\begin{Definition}
With the notation associated to the dissection $\cD$ of $P_n$ as before,
we set
$$\Pol_\cD (x) =
\{ f \in \Z[x_1, \ldots, x_m]\mid \deg_{x_i}f \leq d_i-2
\text{ for all } i\}\:.$$
Then we define the {\em complementing map}
$$\phi_\cD: \Pol_\cD (x) \to \Pol_\cD (x) $$
by defining it on monomials $x^a=\prod_{i=1}^m x_i^{a_i} \in \Pol_\cD (x)$  to be
$$\phi_\cD(x^a)=\prod_{i=1}^m x_i^{d_i-2-a_i}
$$
and linear extension.

If the dissection $\cD$ is clear from the context, we will also simply
write $\bar f$ instead of~$\phi_\cD(f)$, for  $f\in \Pol_\cD (x)$.
\end{Definition}

\smallskip

In proofs on dissected polygons
we will often use the fact that a nontrivial dissection
always has  a {\em boundary $d$-gon}, i.e.,
a $d$-gon $\al$ where precisely one of the boundary edges
of $\al$ is an interior diagonal and the other
are on the boundary of the dissected polygon; indeed, an easy induction proof
shows that there are always at least two such boundary pieces
in any nontrivial dissection.
In a proof by induction on the number of pieces in a dissection,
$\al$ is then a $d$-gon glued onto a boundary edge of a dissected
polygon with fewer pieces.
This strategy is used for the following result which generalizes
the symmetry property in the arithmetical case;
in fact, the combinatorial arguments used in the arithmetical situation
can be applied with a little extra care here as well.

\begin{Theorem} \label{thm:complementarysymmetry}
Let $n\ge 3$, let $\cD$ be a dissection of
$P_n$ into $m$ pieces,
and let $W_\cD(x_1,\ldots,x_m)=(w_{i,j})$ be the associated
weight matrix for the walks around~$P_n$. Then
$$\overline{w_{i,j}} = w_{j,i}  \qquad
\text{ for all } 1\leq i,j\leq n \:,$$
i.e.,
$W_\cD$ is a complementary symmetric matrix:
$$W_\cD^t=\overline W_\cD=(\overline{w_{i,j}}) \:.$$
\end{Theorem}

\begin{proof}
We prove the result by induction on~$m$.
For $m=1$, we have just an $n$-gon $\al=\al_1=P_n$ (with no diagonals).
The only walk from any vertex $i$ to a different vertex $j$ is
$p_{i+1}=\al,\al,\ldots,p_{j-1}=\al$, hence
we have $w_{i,j}w_{j,i}=x_1^{n-2}$
for all $i\neq j$,
coming from the choice of $\al$ for each of
the $n-2$ vertices different from $i,j$.
Thus $\overline{w_{i,j}}=w_{j,i}$.

Now assume that we have already proved the claim for dissections of
polygons with at most $m-1$ pieces.
We consider a dissection~$\cD$ of a polygon $\cP$ into $m>1$ pieces
and want to prove complementary symmetry for its weight
matrix $W_\cD(x)=(w_{i,j})$.
Let $\al=\al_m$ be a boundary $d$-gon of the dissection
$\cD$ of $\cP=P_{n+d-2}$;
w.l.o.g.\ we may assume that its edge in the interior of $\cP$
is a diagonal between the vertices $1$ and~$n$,
so that $\al$ has vertices $1,n,n+1,\ldots, n+d-2$, as in the
following figure.
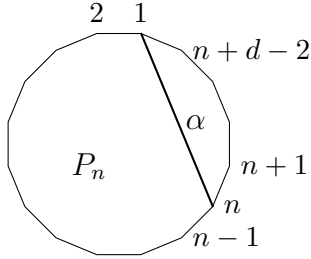
\[
  \begin{tikzpicture}[auto]
    \node[name=s, shape=regular polygon, regular polygon sides=16, minimum size=3cm, draw] {};
 \draw[thick] (s.corner 1) to (s.corner 12);
  \draw[shift=(s.corner 14)]  node[left=5pt]  {$\al$};
  \draw[shift=(s.corner 6)]  node[right=20pt]  {$P_n$};
  \draw[shift=(s.corner 1)]  node[above]  {{\small 1}};
  \draw[shift=(s.corner 2)]  node[above]  {{\small 2}};
  \draw[shift=(s.corner 12)]  node[right]  {{\small $n$}};
  \draw[shift=(s.corner 13)]  node[right]  {{\small $n+1$}};
  \draw[shift=(s.corner 16)]  node[right]  {{\small $n+d-2$}};
  \draw[shift=(s.corner 11)]  node[right]  {{\small $n-1$}};
  \end{tikzpicture}
\]
We denote by $P_n$ the $n$-gon with vertices $1,2,\ldots,n$
obtained from $\cP=P_{n+d-2}$ by cutting off~$\al$;
let $\cD'$ be the dissection of $P_n$ obtained from $\cD$
by removing~$\al$, and let
$W'=W_{\cD'}(x')=W_{\cD'}(x_1,\ldots,x_{m-1})=(w_{i,j}')_{1\le i,j\le n}$
be the corresponding weight matrix.
We already know that $w_{i,i}=0$ and
$w_{i,i}'=0$ for all vertices $i$ in the respective dissected polygons.

{\em Case~1:}
Let $i,j \in \{1,\ldots,n\}$ with $i< j$.
We clearly have $w_{i,j}=w_{i,j}'$, as the walks considered
and their weights  are the same in both dissections.
\smallskip

On the other hand,  there is a bijection
mapping walks from $j$ to $i$ in $P_n$ to walks from $j$ to $i$
in $P_{n+d-2}$ given by inserting the $d$-gon $\al$ with multiplicity
$d-2$:
$$(p_{j+1},\ldots,p_n,p_1,\ldots,p_{i-1})
\mapsto  (p_{j+1},\ldots,p_n,p_{n+1}=\al, \al,\ldots, \al, p_{n+d-2}
=\al,  p_1,\ldots,p_{i-1})\:. $$
Note here that the vertices $n+1,\ldots,n+d-2$ are only incident to the $d$-gon $\al$ in the dissection $\cD$.
Thus $w_{j,i}=w_{j,i}' \cdot  x_m^{d-2}$, and hence applying induction we have
$$\phi_\cD(w_{i,j}) = \phi_\cD(w_{i,j}') =
\phi_{\cD'}(w_{i,j}') \cdot x_m^{d-2}
=
w_{j,i}' \cdot x_m^{d-2}
=w_{j,i}$$
as claimed.

\smallskip

{\em Case~2:}
Next we consider two vertices $i\in \{1,\ldots,n\}$ and
$j\in \{n+1,\ldots,n+d-2\}$.
We claim:
$$w_{i,j}=w_{i,n}'\cdot x_m^{j-n} + w_{i,1}'\cdot x_m^{j-n-1}\:.$$
For this, note that a walk from $i$ to $j$ in $P_{n+d-2}$
has the form
$$(p_{i+1}, \ldots, p_n, p_{n+1}=\al, \al, \ldots, p_{j-1}
=\al),
$$
and we distinguish the cases $p_n=\al$ and $p_n\ne \al$.
The walks with $p_n=\al$ correspond bijectively to walks
$(p_{i+1}, \ldots, p_{n-1})$ from $i$ to $n$ in $P_n$,
the ones with $p_n\ne \al$ correspond bijectively
to walks $(p_{i+1}, \ldots, p_n)$ from $i$ to $1$ in $P_n$.
The corresponding weight polynomials give the two contributions
on the right hand side above.

Similarly,
$$w_{j,i}=w_{1,i}'\cdot x_m^{n+d-1-j}+w_{n,i}'\cdot x_m^{n+d-2-j}\:.$$
Here, walks from $j$ to $i$ in $P_{n+d-2}$
have the form $(p_{j+1}=\al,\al \ldots, p_{n+d-2}=\al,p_1,\ldots, p_{i-1})$,
and we distinguish the cases $p_1=\al$ and $p_1\ne \al$.
The ones with $p_1=\al$ correspond bijectively to walks
$(p_2, \ldots, p_{i-1})$ from $1$ to $i$ in $P_n$,
the ones with $p_1\ne \al$ correspond bijectively
to walks $(p_1, \ldots, p_{i-1})$ from $n$ to $i$ in $P_n$.
Again, the corresponding weight polynomials give the two
contributions on the right hand side above.

Now by induction we have complementary symmetry of $W'$,
and thus we get indeed $\phi_\cD(w_{i,j})=w_{j,i}$.

\smallskip

{\em Case~3:}
Finally, we consider two vertices $i,j\in \{n+1,\ldots,n+d-2\}$,
with $i<j$. We have only one walk from $i$ to $j$ around $P_{n+d-2}$,
giving the weight polynomial $w_{i,j}=x_m^{j-i-1}$.
Now consider a walk from $j$ to $i$ around $P_{n+d-2}$;
this has the form
$$
(p_{j+1}=\al, \ldots, p_{n+d-2}=\al, p_1, p_2, \ldots , p_{n-1},p_n, p_{n+1}=\al, \ldots, p_{i-1}=\al)
\,.$$
Then the sequence $(p_2,\ldots,p_{n-1})$ is a walk
from $1$ to $n$ around $P_n$; as noted before,
here necessarily each polygon of degree $s$ of the dissection $\cD'$
appears exactly $s-2$ times,
hence we must have $p_1=\al=p_n$.
Furthermore, we have by induction
$$w_{1,n}'=\phi_{\cD'}(w_{n,1}')=\phi_{\cD'}(1)
=\prod_{k=1}^{m-1}x_k^{d_k-2}\:,$$
and hence the sequence $(p_2,\ldots,p_{n-1})$ is the unique
walk from $1$ to $n$ around $P_n$.
Thus also the walk from $j$ to $i$ in $P_{n+d-2}$ is
unique, giving the weight polynomial
$$w_{j,i}= x_m^{d-1-j+i}\prod_{k=1}^{m-1}x_k^{d_k-2} =
\phi_\cD(x_m^{j-i-1}) = \phi_\cD(w_{i,j})\:.$$

This completes the proof of the complementary
symmetry of the matrix~$W_\cD$.
\end{proof}
\medskip

\begin{Remark}
Note that the transposed matrix $W_\cD^t$ corresponds to
taking the {\em clockwise} walks around the polygon~$P_n$.
From the original definition of walks it is not clear
that we have this nice complementary relationship between the
weights of the counterclockwise and clockwise walks around~$P_n$.
\end{Remark}

\section{Weight matrices: determinant and diagonal form}\label{sec:detandsmith}

Take a dissected polygon $P_n$ and glue a further $d$-gon
onto one of its edges.
Surprisingly, one observes
that the determinant of the weight matrix for the larger polygon
is independent of the chosen gluing edge.
Indeed, we will see that
the determinant of the weight matrix of any dissection $\cD$
of $P_n$ only depends on its dissection type, i.e., on the multiplicities
of pieces in $\cD$ of the same degree.

More precisely, we provide an explicit formula
for this determinant in the following result
on arbitrary polygon dissections.

\begin{Theorem}\label{thm:detformula}
Let $\cD=\{\al_1, \ldots,\al_m\}$ be a dissection of $P_n$
of type $\{d_1, \ldots, d_m\}$,
with associated indeterminates $x_1,\ldots,x_m$, respectively.
Set $c=\prod_{i=1}^m x_i^{d_i-2}$.
Let $W_\cD(x_1,\ldots,x_m)$ be the corresponding weight matrix.
Then
$$
\det W_\cD(x_1,\ldots,x_m) = (-1)^{n-1} \prod_{i=1}^m \sum_{j=0}^{d_i-2} (c x_i^2)^j =(-1)^{n-1} \prod_{i=1}^m \frac{(cx_i^2)^{d_i-1}-1}{c x_i^2 -1} \:.
$$
\end{Theorem}

From the formula above we see immediately
that the determinant only depends on the size of the pieces
in the dissection but it is independent from the way they are glued together.
It is indeed a multivariate polynomial that is invariant under permutation of indeterminates that correspond to pieces in the dissection of the same size.

\medskip

We will obtain the formula above as a special case of an even more general
determinant formula in Theorem~\ref{thm:detandsmith},
where we will also provide an equivalent diagonal form for the weight matrix.
Indeed, our proof strategy will force us to also put indeterminate
weights~$q_1, \ldots, q_n$ on the edges $e_1, \ldots, e_n$ of a polygon $P_n$,
with $q_i$ on the edge $e_i$ between $i$ and $i+1$ (taking the vertex number
modulo $n$).

Then a walk $w$ from $i$ to $j$ around the polygon $P_n$ does not only get its weight from the chosen polygons along the walk,
but we also record a contribution from
the edges  traversed,
i.e., the new weight of such a walk is defined to be the monomial
$$x^w q^w = x^w \prod_{s=i}^{j-1} q_s \:.$$
We then define the new weight matrix
$W_\cD(x;q)=W_\cD(x_1,\ldots,x_m;q_1,\ldots,q_n)$ associated to the dissection $\cD$ of the polygon $P_n$
to be the matrix with $(i,j)$-entry
$$
v_{i,j} =\sum_{w: \text{ walk from $i$ to $j$}} x^w q^w
\quad \text{ for } 1\leq i,j\leq n\:.$$
Clearly, with $W_\cD(x)=(w_{i,j})$ being our previous weight matrix,
the connection to the entries of the new weight matrix is given by
$$v_{i,j}= w_{i,j} \prod_{s=i}^{j-1} q_s \quad \text{ for } 1\leq i,j\leq n\:.$$
Note in particular that $v_{i,i+1}=q_i$, for all~$i$.

\begin{Definition}
Set $\Z[x;q]=\Z[x_1, \ldots, x_m;q_1, \ldots, q_n]$.
We now complement monomials in
$$\Pol_\cD(x;q)= \{f \in \Z[x;q] \mid
\deg_{x_i} f \leq d_i -2 \text{ for all } i,
\deg_{q_j} f \leq 1 \text{ for all } j\}$$
not only with respect to $c=\prod_{i=1}^m x_i^{d_i-2}$
in the indeterminates $x_i$ as before, but also with respect to
$\eps=\prod_{j=1}^n q_j$ in the indeterminates $q_j$, i.e.,
the contribution  $\prod_{j\in I} q_j$ is changed to
the $\eps$-complement
$\prod_{j\not\in I} q_j$, for $I\subseteq \{1,\ldots,n\}$.
Extending this linearly, we define a new complementing map
$$\psi_\cD: \Pol_\cD(x;q)\to \Pol_\cD(x;q)\:.$$
\end{Definition}

By Theorem~\ref{thm:complementarysymmetry} and the definition
of the new weights given above, it follows immediately that the weight matrix
$W_\cD(x;q)$ is complementary symmetric with respect to~$\psi_\cD$, i.e.,
we have

\begin{cor} \label{cor:complementarysymmetry}
Let $n\ge 3$, let $\cD$ be a dissection of
$P_n$ into $m$ pieces,
and let $W_\cD(x;q)=(v_{i,j})$ be the associated
weight matrix. Then
$$v_{j,i}=\psi_\cD(v_{i,j})   \quad
\text{for all } 1\leq i,j\leq n \:,$$
i.e.,
$W_\cD(x;q)$ is complementary symmetric with respect to $\psi_\cD$:
$$W_\cD^t=(\psi_\cD(v_{i,j}))_{i,j} \:.$$
\end{cor}

\medskip

For the new weight matrices $W_\cD (x;q)$
we now give an explicit formula for the determinant,
and we also find an equivalent diagonal form.

\begin{Theorem}\label{thm:detandsmith}
Let $\cD=\{\al_1, \ldots,\al_m\}$ be a dissection of $P_n$
of type $\{d_1, \ldots, d_m\}$,
with associated indeterminates $x_1,\ldots,x_m$, respectively,
and indeterminate weights $q_1, \ldots,q_n$ on the edges $e_1, \ldots,e_n$.
Let $R=\Z[x_1^\pm,\ldots,x_m^\pm;q_1^\pm,\ldots,q_n^\pm]$
be the ring of Laurent polynomials in these indeterminates.
Set $c=\prod_{i=1}^m x_i^{d_i-2}$
and $\eps=\prod_{i=1}^n q_i$.

Let $W_\cD (x;q)=W_\cD (x_1,\ldots,x_m;q_1,\ldots,q_n)$ be the  weight matrix corresponding to~$\cD$.
Then
$$
\det W_\cD(x;q)
= (-1)^{n-1} \eps \prod_{i=1}^m \sum_{j=0}^{d_i-2} (\eps c x_i^2)^j
= (-1)^{n-1} \eps \prod_{i=1}^m \frac{ (\eps c x_i^2)^{d_i-1}-1}{\eps c x_i^2-1}\:.
$$
Furthermore, there are matrices $P,Q \in \GL(n,R)$ such that
$$P\cdot W_\cD (x,q) \cdot Q = \Delta(\sum_{j=0}^{d_1-2} (\eps c x_1^2)^j, \ldots, \sum_{j=0}^{d_m-2} (\eps c x_m^2)^j, 1, \ldots , 1)\:.$$
\end{Theorem}

We will prove this formula
by induction on the number of pieces of the dissection.
The transformations giving the matrices $P,Q$
will be constructed quite explicitly.

For the start of the induction proof
we first need to consider the weight matrix
for a polygon $P_d$ without any dissecting diagonals
but with edges weighted by indeterminates.
The following result is exactly the determinant formula
expected in the situation where we have
a trivial dissection with just the original polygon.

\begin{Proposition}\label{prop:polygondet}
Let $d\geq 3$,  $P_d$ a convex $d$-gon,
and let its edges be weighted by $q_1, \ldots , q_d$
as before; set $\eps=\prod_{i=1}^d q_i$.
Let $W_d(x;q)=W_d(x;q_1,\ldots,q_d)$
be the  weight matrix of~$P_d$ for the trivial
dissection $\cD=\{P_d\}$.
Then
$$
\det W_d(x;q)
= (-1)^{d-1} \sum_{j=0}^{d-2}  x^{dj} \prod_{i=1}^d q_i^{j+1}
= (-1)^{d-1} \eps \sum_{j=0}^{d-2}  (\eps x^d)^j \:.
$$
More precisely, with $R=\Z[x^{\pm};q_j^{\pm},j=1,\ldots, n]$
being the ring of Laurent polynomials,
we find $P\in \GL(d,R)$ with $\det P=(-1)^{d-1}$ and $Q\in \SL(d,R)$
such that
$$P \cdot W_d(x;q) \cdot Q=
\Delta(q_1,\ldots,q_{d-1},q_d\sum_{j=0}^{d-2}  (\eps x^d)^j)\:.$$
Thus we have $U,V\in \GL(d,R)$ with
$$U\cdot W_d(x;q) \cdot V = \Delta(\sum_{j=0}^{d-2}  (\eps x^d)^j,1,\ldots,1)\:.$$
\end{Proposition}

\begin{proof}
The matrix $W_d(x;q)$ has the following form
$$
\begin{pmatrix}
0 & q_1 & q_1q_2x & q_1q_2q_3x^2 &  \cdots && q_1\cdots q_{d-1}x^{d-2}\\
q_2\cdots q_d x^{d-2} & 0 & q_2 & q_2q_3x & \cdots && \vdots \\
q_3\cdots q_d x^{d-3} & q_3\cdots q_dq_1 x^{d-2} & 0 & q_3\\
\vdots & \vdots &&  &  & \ddots & \vdots  \\
q_{d-1}q_d x & q_{d-1}q_dq_1 x^2& & \cdots & & & q_{d-1}\\
q_d & q_dq_1x & & \cdots & & q_dq_1\cdots q_{d-2}x^{d-2} & 0
\end{pmatrix}
$$
In a first step, working from right to left, we subtract $q_{i-1}x$ times
column $i-1$ from column~$i$, for $i=d,d-1,\ldots,2$.
The transformed matrix is then
$$
\begin{pmatrix}
0 & q_1 & 0 & 0 & \cdots  &  0\\
q_2\cdots q_d x^{d-2} & -\eps x^{d-1} & q_2 & 0 & \cdots & 0 \\
q_3\cdots q_d x^{d-3} & 0 & -\eps x^{d-1} & q_3\\
\vdots & & & \ddots & \ddots & \vdots\\
q_{d-1}q_d x & 0 & & &  -\eps x^{d-1} & q_{d-1}\\
q_d & 0 &  & \cdots  &  0 & -\eps x^{d-1}
\end{pmatrix}
$$
Using columns~$3, \ldots , d$ in succession to transform all entries
in the first column except the last one to zero,
we finally reach a situation where the only non-zero
entry in the first column is $q_d \sum_{j=0}^{d-2} (\eps x^d)^j$,
at the bottom.
We now use  rows~$1,\ldots,d-1$ in succession to transform the
entries just below each $q_i$ to zero, for $i=1,\ldots, d-1$.
Moving the first column to the end via a permutation matrix $P$,
of determinant $(-1)^{d-1}$, we obtain the diagonal matrix
$$\Delta(q_1,\ldots,q_{d-1},q_d\sum_{j=0}^{d-2}  (\eps x^d)^j)$$
as claimed.
In particular, since all transformations apart from
the last permutation were unimodular,
we have also proved the stated determinant formula.
\end{proof}

There are a couple of consequences which we want to state explicitly.
With the result above we have computed the determinants
of special Toeplitz matrices which arise in our context
as weight matrices
of special edge-weighted polygons.

\begin{Corollary}
Let $d\geq 3$, $m\in \{0, \ldots,d-2\}$.
Let $T_d^{d-m}(x,q)$ be the $(d-m)\times (d-m)$
Toeplitz matrix with entries
$$
t_{i,j}=
\begin{cases}
x^{j-i} & \text{if } j>i\\
0 & \text{if } i=j\\
qx^{d-1-(i-j)} & \text{if } i>j
\end{cases}
$$
for $i,j\in \{1,\ldots,d-m\}$.
Then
$$
\det T_d^{d-m}(x,q) = (-1)^{d-m-1} q x^m \sum_{j=0}^{d-2-m} (qx^d)^j
%= (-1)^{d-m-1} qx^m \frac{(qx^d)^{d-1-m}-1}{qx^d-1}
\:.
$$
\end{Corollary}

\proof
We have
$$T_d^{d-m}(x,q)=W_d(x;1,\ldots,1,qx^m)\,
$$
i.e., this is the weight matrix for the walks around a
$(d-m)$-gon of weight~$x$, where
the edge weights  are specialized to~1, except for the edge joining $d-m$ and $1$ which is set to~$qx^m$.
Now the result follows immediately from the determinant formula
for the weight matrix.
\qed

\smallskip

In particular, for $m=0$ we have a well-known result on special circulant matrices, obtained here over an arbitrary field and without considering  eigenvalues:
\begin{Corollary}
Let $P_d$ be a convex $d$-gon with
associated indeterminate $x$,
indeterminate weight $q$ on the edge~$e_1$,
and weight~1 on all other edges.
Then
$$\det W_d(x;q,1,\ldots,1)=
(-1)^{d-1}\sum_{j=0}^{d-2} q^{j+1}x^{dj}
\:.$$
In particular,
$$\det W_d(x)= (-1)^{d-1}\sum_{j=0}^{d-2} x^{dj}
\:.$$
\end{Corollary}

\medskip

Before we move on to the proof of Theorem~\ref{thm:detandsmith}, we provide a
matrix result that will be useful.

\begin{Proposition}\label{prop:matrix}
Let $R$ be a commutative ring (with~1), $s\in \N$.
Let $y\in R$, and let $u_1, \ldots, u_{s-1}$ be units in~$R$;
set $\delta = \prod_{i=1}^{s-1}u_i$.
Define the $d\times d$ matrix $U$ over $R$ by
$$
U=
\begin{pmatrix}
1+y & u_1y & u_1u_2y & \cdots && u_1\cdots u_{s-1}y\\
u_1^{-1} & 1+y & u_2y & u_2u_3y && u_2\cdots u_{s-1}y\\
(u_1u_2)^{-1} & u_2^{-1} & 1+y &&& u_3\cdots u_{s-1}y\\
\vdots &&& \ddots \\
&&&& \ddots & u_{s-1}y\\
(u_1 \cdots u_{s-1})^{-1} & (u_2\cdots u_{s-1})^{-1} && (u_{s-2}u_{s-1})^{-1}& u_{s-1}^{-1} & 1+y
\end{pmatrix}\:.
$$
Then there are
$P\in \GL(s,R)$ with $\det P=(-1)^{s-1}$ and $Q\in \SL(s,R)$
such that
$$P \cdot U \cdot Q=
\Delta(-u_1,\ldots,-u_{s-1},\delta^{-1}\sum_{j=0}^{s} y^j)\:,$$
and in particular,
$$\det U = \sum_{j=0}^{s} y^j\:.$$
\end{Proposition}

\begin{proof}
We use similar transformations as in the proof of Proposition~\ref{prop:polygondet}.
For $i=s-1,s-2,\ldots,2$, we multiply
column~$i-1$ by $u_{i-1}$ and subtract this from column~$i$;
then the new column~$i$ has at most two non-zero entries, namely $-u_{i-1}$ and $y$, in rows $i-1$ and $i$.
We then use all these columns in turn to transform the entries in the first column to zero, except the last one, which is transformed to
$\delta^{-1}\sum_{j=0}^s y^j$.
The unit entries $-u_i$ can then be used to turn all entries $y$ just below them to zeros. Moving the first column to the final column then
produces the asserted diagonal matrix, via a permutation matrix of determinant $(-1)^{s-1}$. All other transformations were unimodular, hence the claim about
the determinant also follows.
\end{proof}

\bigskip

We can now embark on the {\bf proof of Theorem~\ref{thm:detandsmith}}.

As we have proved the result in the case of an arbitrary polygon with trivial
dissection, i.e., a dissection with only one piece,
we can now assume that we have a dissection $\cD$ of a polygon $\cP$
with $m>1$ pieces,
and that the result holds for dissections with fewer than $m$ pieces.

As before, we let $\al=\al_m$ be a boundary $d$-gon of the dissection
$\cD$ of $\cP=P_{n+d-2}$ (say), with its interior edge a diagonal between
the vertices $1$ and~$n$,
so that $\al$ has vertices $1,n,n+1,\ldots, n+d-2$, as pictured
in the proof of Theorem~\ref{thm:complementarysymmetry}.

First, we transform the weight matrix $W_\cD(x;q)=W=(v_{i,j})$ into a
block diagonal form as follows.

For $1\le i \le n$ and $n+1\le j \le n+d-2$, sorting the
walks $w=(p_{i+1},\ldots ,p_{j-1})$ from $i$ to $j$
according to whether $p_n=\al$ or not, we have
$$
v_{i,j}=v_{i,n} \cdot q_n\cdots q_{j-1}x_m^{j-n}
+v_{i,1} \cdot (q_j\cdots q_{n+d-2}x_m^{n+d-1-j})^{-1}\:.
$$
Similarly, due to the complementary symmetry, we have
$$
v_{j,i}=v_{1,i} \cdot q_j\cdots q_{n+d-2}x_m^{n+d-1-j}
+v_{n,i} \cdot (q_n\cdots q_{j-1}x_m^{j-n})^{-1}
\:.
$$
Now subtract $(q_j\cdots q_{n+d-2}x_m^{n+d-1-j})^{-1}$ times column~1
plus $q_n\cdots q_{j-1}x_m^{j-n}$ times column~$n$ from column~$j$,
for $n+1\le j \le n+d-2$.
By the equation for $v_{i,j}$ above, this produces the zero matrix in the upper right
$n\times (d-2)$ block of~$W$;
analogous row operations lead to a zero matrix also
in the lower left $(d-2)\times n$ block of~$W$.
Hence we have transformed $W$ into the block diagonal sum of
the upper left $n\times n$-part $\tilde W$ of our weight matrix $W$,
and a $(d-2)\times (d-2)$ matrix $W'$, i.e., to the form
$
\begin{pmatrix}
 \tilde W & \bf{0} \\
\bf{0}  & W'
\end{pmatrix}
$,
with $\tilde W=(v_{i,j})_{1\le i,j\le n}$ and zero matrices of suitable size.

Considering the matrix $\tilde W$,
one observes that this is the weight matrix $W_{\tilde\cD}(x;\tilde q)$
for the dissection of an $n$-gon $P_n$
which arises from the given dissection $\cD$ of the polygon $P_{n+d-2}$
in the following way:
we  delete the vertices $n+1,\ldots,n+d-2$ of $\cP$,
we give the edge $\tilde e_n$ from vertex $n$ to $1$ in the $n$-gon $P_n$
on the vertices $1,\ldots,n$
the weight $\tilde q_n = q_n\cdots q_{n+d-2}x_m^{d_m-2}$,
leave all other edge weights unchanged,
and we obtain the dissection $\tilde\cD$ of $P_n$ by removing
$\al$ from~$\cD$.
Thus, by induction we know that the result holds for $\tilde W$.
We note that the weight $\tilde q_n$ just fits to give
for the parameters $\tilde\eps$, $\tilde c$
of the dissection $\tilde \cD$:
$$
\tilde \eps \tilde c = \prod_{i=1}^n \tilde q_i \prod_{k=1}^{m-1} x_k^{d_k-2}
 =\prod_{i=1}^{n+d-2} q_i \prod_{k=1}^{m} x_k^{d_k-2} = \eps c\:.
$$
Hence with suitable matrices
$\tilde P, \tilde Q \in \GL(n,R)$ we have
$$
\tilde W' =
\tilde P \cdot \tilde W \cdot \tilde Q =
\Delta(\sum_{j=0}^{d_1-2} (\eps c x_1^2)^j, \ldots, \sum_{j=0}^{d_{m-1}-2} (\eps c x_{m-1}^2)^j, 1, \ldots , 1)\:,$$
and
$$\det \tilde W=
(-1)^{n-1} \tilde \eps \prod_{i=1}^{m-1} \sum_{j=0}^{d_i-2} (\eps c x_i^2)^j
\:.$$

Now we turn to the $(d-2)\times (d-2)$ matrix~$W'=(v_{i,j}')_{n+1\le i,j\le n+d-2}$ that appeared
in the lower right block after the transformation of $W$ into block diagonal form.
For the entries of this matrix we have
(due to the transformations on the columns)
$$
v_{i,j}' = v_{i,j}- (q_j\cdots q_{n+d-2}x_m^{n+d-1-j})^{-1} v_{i,1}
- q_n\cdots q_{j-1}x_m^{j-n} v_{i,n} \:.
$$
As we now consider vertices $i,j$ both belonging to $\alpha$,
the values $v_{i,j}$, $v_{i,1}$, $v_{i,n}$ can easily be given explicitly.
We have for $n+1\le i \le n+d-2$
$$
v_{i,1}=(\prod_{k=i}^{n+d-2}q_k) x_m^{n+d-2-i}, \;
v_{i,n}=\frac{\eps c}{(\prod_{k=n}^{i-1}q_k) x_m^{i-n-1}}\:,
$$
and for $n+1\le i<j \le n+d-2$
$$v_{i,j}= (\prod_{k=i}^{j-1} q_k) x_m^{j-i-1}. $$
Inserting this into the equation for $v_{i,j}'$ above, one
obtains the first of the following equations,
and the expressions in the other two cases can be simplified similarly:
$$v_{i,j}'=
\left\{
\begin{array}{lcl}
-\eps c q_i\cdots q_{j-1}x_m^{j-i+1} & \text{for } i<j, \\[6pt]
-x^{-1}(1+\eps c x_m^2) & \text{for } i=j, \\[6pt]
-(q_j\cdots q_{i-1}x_m^{i-j+1})^{-1} & \text{for } i>j.
\end{array}
\right.
$$
Now the matrix $-x_m^{-1}W'$ has the form of the matrix in Proposition~\ref{prop:matrix},
with $R$ our ring of Laurent polynomials,
and parameters $s=d-2$, $y=\eps c x_m^2$, $u_i=q_{n+i}x_m$, $i=1,\ldots,s-1$.
Hence we have transformation matrices $P'\in \GL(d-2,R)$, $\det P'=(-1)^{d-1}$,
and $Q'\in \SL(d-2,R)$, such that
$$
P'\cdot W' \cdot Q' =
\Delta(q_{n+1}x_m^2, \ldots, q_{n+d-3}x_m^2,-x_m \prod_{i=n+1}^{n+d-3}(q_ix_m)^{-1}\sum_{j=0}^{d-2}(\eps cx_m^2)^j)\:.
$$
For our theorem, we allow arbitrary transformation matrices in $\GL(d-2,R)$,
and with a suitable monomial transformation matrix in $\GL(d-2,R)$ we arrive at the diagonal matrix $\Delta(\sum_{j=0}^{d-2}(\eps cx_m^2)^j,1, \ldots,1)$.
Together with the diagonal $n\times n$ matrix $\tilde W'$
from above (sitting in the upper left corner), and then a final sorting,
this gives the claimed diagonal form for~$W$; indeed, following the
inductive steps, the transforming matrices are constructed explicitly
along the way.

For the determinant of $W'$, we have from the
previous equation
$$
\det W'= (-1)^{d-2} x_m^{-(d-2)}\sum_{j=0}^{d-2}(\eps cx_m^2)^j\:.
$$
As $\tilde \eps = \eps \cdot x_m^{d-2}$, this implies
$$
\det W = \det \tilde W \cdot \det W' =
(-1)^{n+d-3} \eps \prod_{i=1}^{m} \sum_{j=0}^{d_i-2} (\eps c x_i^2)^j \:,
$$
which is the assertion we wanted to prove.
Thus we are done.
\qed

\section{Generalized polynomial friezes}\label{sec:friezes}

As a further nice feature,
we can generalize the crucial local determinant condition
of Conway-Coxeter frieze patterns
and get a detailed geometric picture of the dissection
from the associated generalized frieze.

The {\em generalized (polynomial) frieze pattern}  associated
to a dissected polygon $P_n$ with dissection $\cD$
is a periodic pattern of $n-1$ interlaced rows as before where the numbers $m_{ij}$ are now replaced by the weight polynomials $v_{ij}$.

As the edge weights $q_i$ are easy to control, here is an example where
we have specialized the edge weights to~1 and
just give the weight contributions from the pieces;
it is associated to the triangulated pentagon pictured in Section~2:

\smallskip

$$\small
\begin{array}{ccccccccccccccc}
 1 & & 1 & & 1 & & 1 & & 1 & & 1 &
 \ldots \\
 \ldots & a & & a+b+c & & c & & b+c & & a+b & & a & & \\
  ab & & a(b+c) & & (a+b)c & & bc & & a(b+c)+bc & & ab &
  \ldots
  \\
 \ldots & abc & & abc & & abc & & abc & & abc & & abc & &
 \\
\end{array}
$$

\smallskip

In the example, we can already observe that the local $2\times 2$ determinants, i.e., the $2\times 2$-minors
of the weight matrix $W_\cD$ (including the ones
between the last and first column)
are 0 or special monomials.
\medskip

We have to introduce a little bit more notation.
Since the rows and columns of the matrix $W_\cD=(v_{i,j})$ are indexed by
the vertices of the dissected $n$-gon $P_n$
(in counterclockwise order), any $2\times 2$-minor
corresponds to a pair of boundary edges, say
$e=e_i=(i,i+1)$ and $f=e_j=(j,j+1)$. Then the corresponding minor has the form
$$d(e,f) := d_\cD (e,f) := \det \left(
\begin{array}{cc} v_{i,j} & v_{i,j+1} \\ v_{i+1,j} & v_{i+1,j+1}
\end{array} \right).
$$

We now have a generalization of the corresponding theorem
on generalized (arithmetical) friezes in~\cite{BHJ}, i.e.,
we compute all the determinants $d(e,f)$, and we see that we have a
condition for nonzero determinants based on
a ``zig-zag" connection via diagonals between
the two boundary edges under consideration
(see part (c) of Theorem~\ref{thm:zigzag} below for the precise definition).
If there is such a zig-zag connection from
$e=(i,i+1)$ to $f=(j,j+1)$,
those pieces of the dissection which have
at most one vertex on the
counterclockwise route from $j+1$ to $i$ on the polygon
play a special r\^ole:
the determinant $d(e,f)$
is a monomial to which only the edges $e_i,\ldots, e_j$ and these
``zig" pieces contribute; see below for the precise statement
and also for an example.

\begin{Theorem} \label{thm:zigzag}
Let $\cD=\{\al_1, \ldots,\al_m\}$ be a dissection of $P_n$
of type $\{d_1, \ldots, d_m\}$,
with associated indeterminates $x_1,\ldots,x_m$, respectively.
Set $\eps=\prod_{k=1}^n q_k$, $c=\prod_{l=1}^m x_l^{d_l-2}$.
Then the following holds for the $2\times 2$-minors of $W_\cD$
associated to boundary edges $e=(i,i+1)\neq f=(j,j+1)$ of $P_n$:
\begin{enumerate}
\item[{(a)}] $d(e,e)=-\eps c $.
\item[{(b)}] $d(e,f)\neq 0$ if and only if
there exists a sequence
$$e=z_0,z_1,\ldots,z_{s-1},z_{s}=f$$
with $s\in \N_0$, where $z_1, \ldots, z_{s-1}$ are diagonals of the dissection,
such that the following holds for every $k\in\{0,1,\ldots,s-1\}$:
\begin{itemize}
\item[{(i)}] $z_k$ and $z_{k+1}$ belong to a common piece $p_k \in \cD$;
\item[{(ii)}] the pieces $p_0, \ldots , p_{s-1}$ are pairwise different;
\item[{(iii)}] $z_{k}$ is incident to $z_{k+1}$.
\end{itemize}
\smallskip

We call such a sequence a {\em zig-zag sequence} from $e$ to $f$.
If there is such a sequence then those pieces of~$\cD$ which have
at most one vertex on the counterclockwise route from $j+1$ to $i$
around $P_n$ are called {\em zig pieces} for $(e,f)$
(the others are the {\em zag pieces}).
\smallskip

\item[{(c)}]
If there exists a zig-zag sequence from $e$ to $f$, then
$$d(e,f)= q_iq_j\prod_{k=i+1}^{j-1} q_k^2 \prod_\be x_\be^{2(d_\be -2)} \:,$$
where $k$ runs over all vertices  in the counterclockwise
walk from  $i+1$ to $j-1$ around~$P_n$,
and $\be$ runs over all zig pieces for $(e,f)$.
\end{enumerate}
\end{Theorem}
\medskip

\begin{Example}
Before we embark on the proof,
we consider the following polygon dissection as an example.
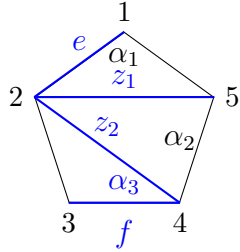
\[
  \begin{tikzpicture}[auto]
    \node[name=s, shape=regular polygon, regular polygon sides=5, minimum size=2.5cm, draw] {};
     \draw[thick,blue] (s.corner 1) to node[above=1pt] {$e$} (s.corner 2);
      \draw[thick,blue] (s.corner 3) to node[below=2pt] {$f$}  node[above=-1pt]  {{$\al_3$}}(s.corner 4);
    \draw[thick,blue] (s.corner 2) to node[above=1pt] {$z_2$} (s.corner 4);
    \draw[thick,blue] (s.corner 2) to node[above=-1pt] {$z_1$} (s.corner 5);
    \draw[shift=(s.corner 1)]  node[above]  {{\small 1}};
    \draw[shift=(s.corner 1)]  node[below=2pt]  {{$\al_1$}};
  \draw[shift=(s.corner 2)]  node[left]  {{\small 2}};
  \draw[shift=(s.corner 3)]  node[below]  {{\small 3}};
  \draw[shift=(s.corner 4)]  node[below]  {{\small 4}};
  \draw[shift=(s.corner 4)]  node[above=16pt]  {{$\al_2$}};
  \draw[shift=(s.corner 5)]  node[right]  {{\small 5}};
  \end{tikzpicture}
\]
Note that we have here
$$
d(e,f)=\det \left(
\begin{array}{cc} q_1q_2(x_1+x_2+x_3) & q_1q_2q_3(x_1+x_2)x_3 \\
q_2 & q_2q_3x_3
\end{array} \right).
$$
We now illustrate the computation of  the determinant via the theorem.
The sequence $e,z_1, z_2,f$ is a zig-zag sequence from $e$ to $f$.
Here, we have only one zig piece for $(e,f)$, namely $\al_3$,
that contributes to the monomial $d(e,f)$.
Hence
$$d(e,f)=q_1q_3q_2^2x_3^2 \:.$$
From $f$ to $e$, we have the zig-zag sequence $f,z_2,z_1,e$,
and here the contributing zig pieces for $(f,e)$ are $\al_1$ and $\al_2$.
Hence
$$d(f,e)=q_1q_3q_4^2q_5^2 x_1^2 x_2^2 \:.$$
\end{Example}
\medskip

\begin{proof} {\em (Theorem~\ref{thm:zigzag})}
Let the edges $e=(i,i+1)$ and $f=(j,j+1)$ be given.
We first consider the effect of boundary pieces between $i$ and $j$ on the minor
$d(e,f)$.
In addition to the weights already associated to the edges of $P_n$ and the pieces of $\cD$, each diagonal $t$ of the dissection gets the weight $q_t=1$.

Assume that $\al\in \cD$ is a boundary piece with all
its vertices between $j+1$ and $i$ (in counterclockwise order).
Let $\mathcal P'$ be the polygon obtained by removing $\al$ from~$\mathcal P$,
and let $\cD'=\cD\setminus\al$ be the corresponding dissection;
we keep (here and in similar situations below)
the same labels for the vertices
(rather than renumbering them in the smaller polygon).
Then
$$d_\cD (e,f) = \det \left(
\begin{array}{cc} v_{i,j} & v_{i,j+1} \\ v_{i+1,j} & v_{i+1,j+1}
\end{array} \right)
=d_{\cD'}(e,f)\:,
$$
as the relevant entries in the weight matrices $W_\cD$ and
$W_{\cD'}$ are exactly the same.

Now let $\be\in \cD$ be a boundary piece with all
its vertices (strictly) between $i$ and $j+1$ (in counterclockwise order),
cut off by the diagonal $t$ (say), and let
$q_\be$ be the product of the weights of the edges $e_k$ belonging to~$\be$.
Let $\mathcal P'$ be the polygon obtained by removing $\be$ from $\mathcal P$,
and let $\cD'=\cD\setminus\be$ be the corresponding dissection.
Now take $r\in \{i,i+1\}$, $s\in \{j,j+1\}$. Let $v_{r,s}$ and $v_{r,s}'$ be the
corresponding entries in the weight matrices $W=W_\cD=(v_{u,v})$ and
$W'=W'_{\cD'}=(v_{u,v}')$, respectively,
where the weights on the edges and pieces are inherited and the
boundary edge $t$ of $\mathcal P'$ has weight $q_t=1$ as stated above;
let $d_{\cD'}'(e,f)$ be the minor to $(e,f)$ in $W'$.
Considering the walks from $r$ to $s$, one easily sees that
$$v_{r,s}=v_{r,s}' \, q_\be \, x_{\be}^{d_\be-2}$$
and thus
$$d_\cD(e,f)=d_{\cD'}'(e,f) \, x_\be^{2(d_\be -2)} q_\be^2 \:.$$
Hence we may cut off boundary pieces $\be$ from $\cD$ on the way from $i+1$ to $j$, recording their contribution $x_\be^{2(d_\be -2)} q_\be^2$ as a factor to the product
as just stated, while on the way from $j+1$ to $i$ the contribution
of boundary pieces is just~1.

\medskip

(a) If $e=(i,i+1)$ then by Theorem~\ref{thm:complementarysymmetry}  we have
$$d(e,e) = \det \begin{pmatrix} v_{i,i} & v_{i,i+1} \\ v_{i+1,i} & v_{i+1,i+1}
\end{pmatrix}
= \det \begin{pmatrix} 0 & q_i \\ \bar q_i & 0 \end{pmatrix} = -\eps c.
$$
\smallskip

(b) and (c) Let $e=(i,i+1)$ and $f=(j,j+1)$ be different boundary edges
of $P_n$.

{\it Case~1:} $e$ and $f$ belong to a common piece $\al$ of $\cD$.

Then, if there is a zig-zag sequence of diagonals,
$e$ and $f$ must be incident,
as the pieces along the zig-zag sequence have to be pairwise
different by condition (ii).

First assume that their common vertex is $j=i+1$.
Since $e$ and $f$ belong to $\al$,
by Definition~\ref{def:weights} we get
$$d(e,f) = \det \begin{pmatrix} v_{i,j} & v_{i,j+1} \\ v_{i+1,j} & v_{i+1,j+1}
\end{pmatrix}
= \det \begin{pmatrix} q_i &  v_{i,j+1} 
\\ 0 & q_{i+1} \end{pmatrix} =
q_iq_{i+1},
$$
as claimed in the statement of the theorem.
The situation where $i=j+1$ is the common vertex of $e$ and $f$
is complementary to the previous case, i.e.,
$$d(e,f) = \det \begin{pmatrix} v_{i,j} & v_{i,j+1} \\ v_{i+1,j} & v_{i+1,j+1}
\end{pmatrix}
= \det \begin{pmatrix} \bar q_j & 0 \\
v_{i+1,j}
& \bar q_i \end{pmatrix} =
\bar q_i \cdot \bar q_j =
q_iq_j \prod_{k=i+1}^{j-1}q_k^2 \prod_{\be\in \cD} x_\be^{2(d_\be -2)}.
$$
In the product over the vertices~$k$, we go around the polygon in counterclockwise order
from $i+1$ to $j-1$. This is the claimed assertion in this case.

Now we assume that $e$ and $f$ are not incident, and we have to show that $d(e,f)=0$.
As argued before, we may successively cut off boundary pieces from the dissection,
and then we only have to show that $d(e,f)=0$ when we are in the situation of two
non-incident boundary edges in a (non-dissected) polygon~$\al$.

In this case we may assume that $1\le i< j-1 < n$ and then
$$d(e,f) = d(i,j)=
\det \begin{pmatrix} q_i \cdots q_{j-1}x_\al^{j-i-1} & q_i\cdots q_j x_\al^{j-i} \\
q_{i+1}\cdots q_{j-1}x_\al^{j-i-2}& q_{i+1}\cdots q_j x_\al^{j-i-1} \end{pmatrix}
 = 0.
$$
Thus we are now done with the situation that $e,f$ belong to a common piece of~$\cD$.
\smallskip

{\em Case~2:} $e$ and $f$ do not belong to a common piece
of~$\cD$.

We have to show that $d(e,f)=0$ when $e,f$ are not connected
by a zig-zag of diagonals,
and otherwise we get a monomial from a zig-zag sequence from $e$ to $f$
where only the zig pieces for $(e,f)$ contribute to the determinant.

Without loss of generality, we may assume $1\le i<j <n$.

We have already discussed at the beginning
of the proof how the removal of a boundary piece
between $i+1$ and $j$ or
between $j+1$ and $i$ (always in counterclockwise order) affects
the determinant.
Indeed, if all the vertices of a boundary piece $\be$ are between
$i+1$ and $j$ ($\be$ is then a zig piece for $(e,f)$),
we get the full contribution $x_\be^{2(d_\be -2)}q_\be^2$
as a factor to the product,
and in the other case where
all the vertices of a boundary piece $\be$ are between $j+1$ and $i$
($\be$ is then a zag piece),
we only get a factor~$1$.
Hence, in our situation we may cut off all such boundary pieces
until we arrive at a minimal convex dissected subpolygon of $P_n$
containing $e$ and~$f$. Via this reduction we may assume that $f=(j,j+1)$
belongs to a boundary piece $\al$ of $P_n$ with internal diagonal
$t$ with endpoints $k$ and $l$, as in the following figure
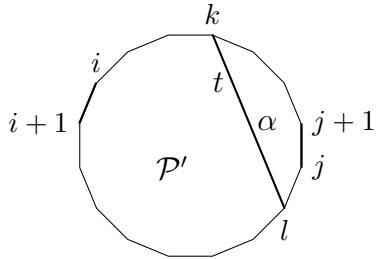
\[
  \begin{tikzpicture}[auto]
    \node[name=s, shape=regular polygon, regular polygon sides=16, minimum size=3cm, draw] {};
 \draw[thick] (s.corner 1) to (s.corner 12);
  \draw[thick] (s.corner 4) to (s.corner 5);
  \draw[thick] (s.corner 13) to (s.corner 14);
  \draw[shift=(s.corner 14)]  node[left=5pt]  {$\al$};
  \draw[shift=(s.corner 4)]  node[right=40pt]  {$t$};
 \draw[shift=(s.corner 6)]  node[right=25pt] {$\mathcal{P'}$} ;
  \draw[shift=(s.corner 1)]  node[above]  {{\small $k$}};
  \draw[shift=(s.corner 12)]  node[below]  {{\small $l$}};
  \draw[shift=(s.corner 13)]  node[right]  {{\small $j$}};
  \draw[shift=(s.corner 14)]  node[right]  {{\small $j+1$}};
  \draw[shift=(s.corner 4)]  node[above]  {{\small $i$}};
  \draw[shift=(s.corner 5)]  node[left]  {{\small $i+1$}};
  \end{tikzpicture}
\]
If $f$ and $t$ are not incident,
then a zig-zag sequence from $e$ to $f$  can not
exist since condition (iii) can not be satisfied.
On the other hand, we have
$v_{i,j+1}=v_{i,j}x_\al q_j$ and
$v_{i+1,j+1}=v_{i+1,j}x_\al q_j$ and hence
$d(e,f)= d(i,j)= 0$.

Now assume that $f$ and $t$ are incident.
First consider the situation where $j=l$.
As before,
let $v_{r,s}'$ denote the weight polynomials for the polygon $\mathcal{P'}$
obtained
from $P_n$ by removing~$\al$, with respect to the dissection
$\cD'=\cD\setminus\al$,
where we recall that the weight $q_t$ of the boundary edge $t$ of $\mathcal{P'}$
is specialized to~1.
For walks from $i$ or $i+1$ to $j+1$ around~$P_n$,
we have to distinguish between the
ones where we choose the piece $\al$ at vertex~$j$, and those where
we use only pieces from $\cD'$ at~$j$.
Then we have
$$
\begin{array}{rcl}
d_\cD(e,f) = d_\cD(i,j)
&=&
\det \begin{pmatrix} v_{i,j} & v_{i,j}x_\al q_j+v_{i,k}'q_j \\
v_{i+1,j} & v_{i+1,j}x_\al q_j+v_{i+1,k}'q_j
\end{pmatrix} \\[20pt]
&=&
\det \begin{pmatrix} v_{i,j}' & v_{i,k}'q_j \\
v_{i+1,j}' & v_{i+1,k}'q_j
\end{pmatrix} \\[20pt]
&=&
q_j d_{\cD'}'(i,k)= q_j d_{\cD'}'(e,t) \:,
\end{array}
$$
where we keep in mind that in the polygon $\mathcal{P'}$ dissected by $\cD'$,
the vertices $j,k$ are indeed successive vertices as endpoints of
the boundary edge~$t$.

Now, if a zig-zag sequence from $e$ to $f$ exists, then
clearly the final diagonal has to be $t$.
Hence a zig-zag sequence from $e$ to $f$
exists in the dissected polygon $\mathcal P$
if and only if there exists a zig-zag
sequence from $e$ to~$t$ in the dissected subpolygon~$\mathcal{P'}$.

Thus, by induction and the formula for the minors above,
we can already conclude that
$d_\cD (e,f)= 0$ unless we find a zig-zag sequence
from $e$ to~$f$.
Furthermore, in the case that we do have a zig-zag sequence
from $e$ to $f$,
we know by induction that
$$d_{\cD'}'(e,t) =
q_iq_t\prod_{r=i+1}^{l-1} q_r^2 \prod_\be x_\be^{2(d_\be -2)}
= q_i\prod_{r=i+1}^{l-1} q_r^2 \prod_\be x_\be^{2(d_\be -2)} $$
where $\be$ runs over the zig pieces for $(e,t)$, i.e.,
having vertices between $i+1$ and~$l$
with the exception of at most one vertex.
Hence
$$d_{\cD}(e,f) =
q_j d_{\cD'}'(e,t)= q_iq_j \prod_{r=i+1}^{j-1} q_r^2 \prod_\be x_\be^{2(d_\be -2)} \:,
$$
with $\be$ still running over the same set of pieces,
as the piece $\al$ is not a zig piece for~$(e,f)$.
Thus the minor is indeed the monomial stated in the theorem.

In the case where $j+1=k$, a similar reasoning shows
$$d_{\cD}(e,f) =
q_j( \prod_{r=l}^{j-1} q_r^2) x_\al^{2(d_\al -2)}  d_{\cD'}'(i,l)
= q_iq_j \prod_{r=i+1}^{j-1} q_r^2 \prod_\be x_\be^{2(d_\be -2)} \:,
$$
where now the piece $\al$ is a zig piece for $(e,f)$ and is thus included
in the set of zig pieces $\be$ contributing to the monomial.

Thus the proof of Theorem~\ref{thm:zigzag} is complete.
\end{proof}

\bigskip

{\bf Acknowledgments.}
The author is grateful to Prof.\ Jiping Zhang for the invitation to present
the results at the
Third International Symposium on
Groups, Algebras and Related Topics
at the
Beijing International Center for Mathematical Research of
Peking University
in June 2013.
Thanks go also to the referee for
useful comments on the first version.

\end{document}